\documentclass[12pt]{amsart}
\usepackage{amssymb,amsmath,amsthm,latexsym,color}
\usepackage{amscd}
\usepackage[all,cmtip]{xy}
\input amssym.def
\input amssym
\newtheorem{theorem}{Theorem}[section]
\newtheorem{corollary}{Corollary}[section]
\newtheorem{lemma}{Lemma}[section]

\newtheorem{remark}{Remark}[section]
\newtheorem{defi}{Definition}[section]
\newtheorem{prop}{Proposition}[section]

\newcommand{\be}{\begin{equation}}
\newcommand{\ee}{\end{equation}}

\renewcommand{\theequation}{\thesection.\arabic{equation}}
\renewcommand{\thetheorem}{\thesection.\arabic{theorem}}
\renewcommand{\theequation}{\thesection.\arabic{equation}}
\begin{document}

\title[] {Embedding of certain vertex algebras without vacuum into 
vertex algebras}

\author{Thomas J. Robinson}

%\thanks{}

\begin{abstract}
We show that certain vertex algebras without vacuum vector may be
embedded into vertex algebras.  The result is a partial analogue of
the simple classical fact that any rng can be embedded into a ring.  A
one-line proof of the case of a vacuum-free vertex algebra (whose
vertex operator map is by definition injective) appeared in \cite{R}
using a powerful result from the representation theory of vertex
algebras as algebras of mutually local weak vertex operators.  Here we
present a more elementary proof of a somewhat more general case.  We
also show that our constructions are canonical.
\end{abstract}

\maketitle

\renewcommand{\theequation}{\thesection.\arabic{equation}}
\renewcommand{\thetheorem}{\thesection.\arabic{theorem}}
\setcounter{equation}{0} \setcounter{theorem}{0}
\setcounter{section}{0}

\section{Introduction}
Vertex algebras were originally mathematically defined by R. Borcherds
in \cite{B}.  The closely related notion of vertex operator algebra
was introduced in \cite{FLM}.  The axioms used in \cite{FLM} were in
``generating function form,'' as we shall use here.  The main axiom of
any vertex-type algebra is the Jacobi identity, or some equivalent.
As a secondary point in \cite{R}, among other things, the notion of
vacuum-free vertex algebra was formalized in order to present a more
detailed examination of equivalent axiom systems for vertex algebras.  A
vacuum-free vertex algebra retains the major axiom, the Jacobi
identity, and thus much of the basic theory of the axioms of vertex
operator algebras can be worked out in this spare setting.  As
mentioned in \cite{R}, the motivation for formalizing vacuum-free
vertex algebras was not example-driven but also, as there, we refer
the reader to \cite{BD} and \cite{HL}, where a vacuum-free setting
appeared.  In this note, we show that certain vertex algebras without
vacuum can be embedded into vertex algebras.

The vertex operator map in any vertex algebra is injective.  This
follows from the ``creation property'' and so is generally not
explicitly stated as an axiom.  It was pointed out in \cite{R}, the
proof having appeared in ``pieces'' in \cite{FHL} and \cite{LL} (see
Remark 2.2.4 in \cite{FHL} and Proposition 3.6.7 in \cite{LL}), that
if injectivity is stated explicitly as an axiom, then the ``vacuum
property'' and creation property axioms each follow from the other in
the presence of the remaining axioms.  Actually, it is further pointed
out in \cite{R} that one only needs to retain ``vacuum-free
skew-symmetry'' in place of the Jacobi identity and the result still
holds (see the introduction to Section 6 and Proposition 6.1 of
\cite{R}).  It is pointed out in Remark 3.6.8 of \cite{LL} that if
injectivity is not separately stated as an axiom, then the creation
property is not redundant and one loses this minor symmetry among the
axioms.  For this reason, the notion of vacuum-free vertex algebras in
\cite{R} retained injectivity as an axiom.  Injectivity was used in
only two other places in that paper, first in the proof of Proposition
4.6, showing that ``weak commutativity'' and vacuum-free skew-symmetry
can replace the Jacobi identity in the definition of the notion of
vacuum-free vertex algebra, and second, in Remark 4.1, which we next
discuss.

The relation between vertex algebras and vacuum-free vertex algebras
and that between rings and rngs (rings without identity element) are,
of course, loosely analogous, a point made in \cite{R}.  It is well
known that rngs can be embedded into rings.  The corresponding
question for vacuum-free vertex algebras was raised by members of the
Quantum Mathematics/Lie Group seminar at Rutgers when a preliminary
version of \cite{R} was presented there, and was again raised by the
referee of \cite{R}.  An easy answer, making use of the injectivity
property was given in Remark 4.1 in \cite{R}.  The proof depended on a
powerful result from the representation theory of vertex algebras as
algebras of mutually local weak vertex operators (see \cite{L};
cf. \cite{LL}).  We shall reproduce this result below, which shows that any
vacuum-free vertex algebra may be embedded into a vertex algebra.  We
then give a (much more involved but somewhat pleasing)
non-representation theoretic proof of somewhat more general results.
This non-representation theoretic proof relies essentially on two
constructions, the first of which is universal, and while the second
one is basis dependent at first, we give a basis-free characterization
of it in the final section.

Addendum: After finishing this paper, I learned of a recent pre-print
\cite{LTW} of Haisheng Li, Shaobin Tan and Qing Wang.  Their paper is
concerned, in part, with some of the topics of this paper.  What we
call an ertex algebra, those authors call a Leibniz vertex algebra and
what we call a $D$-ertex algebra, those authors call a vertex algebra
without vacuum (in the sense of Huang and Lepowsky \cite{HL}).  Our
results in Remark \ref{rem:Deq} and Proposition \ref{prop:Dto1}, an
analog of Theorem \ref{lem:E'ert}, and a canonical property analogous
to our various canonical properties are obtained in \cite{LTW}.  Also,
Corollary 2.27 of \cite{LTW}, which I did not obtain, gives a complete
answer to one of the issues raised here.

We thank James Lepowsky and Haisheng Li for their helpful comments.
\section{Preliminaries}
We shall write $x,y,z,t$ for commuting formal
variables.  In this paper, formal variables will always commute, and
we will not use complex variables.  All vector spaces will be over
$\mathbb{C}$, although one may easily generalize many results to the
case of a field of characteristic $0$.  Let $V$ be a vector space.  We
use the following:
\begin{align*}
V[[x,x^{-1}]]=\biggl\{ \sum_{n \in \mathbb{Z}}v_{n}x^{n}|v_{n} \in V  \biggr\}
\end{align*}
(formal Laurent series),
and some of its subspaces:
\begin{align*}
V((x))=\biggl\{ \sum_{n \in \mathbb{Z}}v_{n}x^{n}|v_{n} \in V, v_{n}=0
\text{ for sufficiently negative } n \biggr\}
\end{align*}
(truncated formal Laurent series),
\begin{align*}
V[[x]]=\biggl\{ \sum_{n  \geq 0}v_{n}x^{n}|v_{n} \in V \biggr\}
\end{align*}
(formal power series), and
\begin{align*}
V[x]=\biggl\{ \sum_{n \geq 0}v_{n}x^{n}|v_{n} \in V, v_{n}=0 \text{ for all
but finitely many } n \biggr\}
\end{align*}
(formal polynomials).

For $f(x)=\sum_{n \in \mathbb{Z}}a_{n}x^{n} \in V[[x,x^{-1}]]$ let
$\text{Res}_{x}:V[[x,x^{-1}]] \rightarrow V$ be given by
\begin{align*}
\text{Res}_{x}f(x)=a_{-1}.
\end{align*}

Further, we shall frequently use the notation $e^{w}$ to refer to the
formal exponential expansion, where $w$ is any formal object for which
such expansion makes sense.  For instance, we have the linear operator
$e^{y\frac{d}{dx}}:\mathbb{C}[[x,x^{-1}]] \rightarrow
\mathbb{C}[[x,x^{-1}]][[y]]$:
\begin{align*}
e^{y\frac{d}{dx}}=\sum_{n \geq
0}\frac{y^{n}}{n!}\left(\frac{d}{dx}\right)^{n}.
\end{align*}
We have (cf. (2.2.18) in \cite{LL}), the \it{automorphism property}\rm:
\begin{align}
\label{defi:aut}
e^{y\frac{d}{dx}}(p(x)q(x))=\left(e^{y\frac{d}{dx}}p(x)\right)
\left(e^{y\frac{d}{dx}}q(x)\right),
\end{align}
for all $p(x) \in \text{\rm End}\,V[x,x^{-1}] $ and $q(x) \in
\text{\rm End}\,V[[x,x^{-1}]]$.  We use the \it{binomial expansion
convention}, \rm which states that
\begin{align}
\label{binexpconv}
(x+y)^{n}=\sum_{k \geq 0}\binom{n}{k}x^{n-k}y^{k},
\end{align}
where we allow $n$ to be any integer and where we define
\begin{align*}
\binom{n}{k}=\frac{n(n-1)(n-2) \cdots (n-k+1)}{k!};
\end{align*}
the binomial expression is expanded in nonnegative powers of the
second-listed variable.  We also have (cf. Proposition 2.2.2 in
\cite{LL}) the \it{formal Taylor theorem}\rm :
\begin{prop}
Let $v(x) \in V[[x,x^{-1}]]$. Then
\begin{align*}
e^{y\frac{d}{dx}}v(x)&=v(x+y).
\end{align*}
\end{prop}
\begin{flushright} $\square$ \end{flushright}

For completeness we include a proof of the following frequently used
 (easy but non-vacuous) fact, which equates two different expansions.
\begin{prop}
\label{prop:assocarith}
For all $n \in \mathbb{Z}$,
\begin{align*}
(x+(y+z))^{n}=((x+y)+z)^{n}.
\end{align*}
\end{prop}
\begin{proof}
If $w_{1}$ and $w_{2}$ are
commuting formal objects, then $e^{w_{1}+w_{2}}=e^{w_{1}}e^{w_{2}}$.
Thus we have
\begin{align*}
(x+(y+z))^{n}
=e^{(y+z)\frac{\partial}{\partial x}}x^{n}
=e^{y\frac{\partial}{\partial x}}
\left(e^{z\frac{\partial}{\partial x}}x^{n}\right)
=e^{y\frac{\partial}{\partial x}}(x+z)^{n}
=((x+y)+z)^{n}.
\end{align*}
\end{proof}
We note as a consequence that for all integers $n$ (and not just
nonnegative integers) we have the (non-vacuous) fact that
\begin{align*}
((x+y)-y)^{n}=(x+(y-y))^{n}=x^{n}.
\end{align*}

We recall the definition of the formal delta function (cf. (2.1.32)
in \cite{LL})
\begin{align*}
\delta (x) = \sum_{ n \in \mathbb{Z}}x^{n}.
\end{align*}

We have the following two elementary identities (cf. Proposition 2.3.8
in \cite{LL}):
\begin{align}
\label{twotermdelta}
x_{1}^{-1}\delta\left(\frac{x_{2}+x_{0}}{x_{1}}\right)-
x_{2}^{-1}\delta\left(\frac{x_{1}-x_{0}}{x_{2}}\right)=0
\end{align}
and
\begin{align}
\label{threetermdelta}
x_{0}^{-1}\delta\left(\frac{x_{1}-x_{2}}{x_{0}}\right)-
x_{0}^{-1}\delta\left(\frac{-x_{2}+x_{1}}{x_{0}}\right)-
x_{1}^{-1}\delta\left(\frac{x_{2}+x_{0}}{x_{1}}\right)=0.
\end{align}

Finally, we recall (see Proposition 2.3.21 and Remarks 2.3.24 and
2.3.25 in \cite{LL}) the delta-function substitution property:
\begin{prop}
\label{prop:deltasub}
For $f(x,y,z) \in \text{\rm End}\,V[[x,x^{-1},y,y^{-1},z,z^{-1}]]$ 
such that for each fixed $v \in V$,
\begin{align*}
f(x,y,z)v \in \text{\rm End}\,V[[x,x^{-1},y,y^{-1}]]((z)),
\end{align*}
and such that
\begin{align*}
\text{\rm lim}_{x \rightarrow y}f(x,y,z)
\end{align*}
exists (where the ``limit'' is the indicated formal substitution), we
 have
\begin{align*}
\delta\left(\frac{y+z}{x}\right)f(x,y,z)=\delta\left(\frac{y+z}{x}
\right)f(y+z,y,z)=\delta\left(\frac{y+z}{x}\right)f(x,x-z,z),
\end{align*}
where, in particular, all the products exist.
\end{prop}
\begin{flushright} $\square$ \end{flushright}

As in \cite{LL}, we use similarly verified substitutions below without
comment.

\section{Definitions and motivating results}
In the spirit of rings without identity being called rngs, we shall
call vertex algebras without vacuum ``ertex algebras'':

\begin{defi} \rm
An \it{ertex algebra} \rm is a vector space equipped, first, with a
linear map (the \it{vertex operator map}) $V \otimes V \rightarrow
V[[x,x^{-1}]]$, \rm or equivalently, a linear map
\begin{align*}
Y(\,\cdot\,,x): \quad &V \, \rightarrow \, (\text{\rm End}V)[[x,x^{-1}]]\\
&v \, \mapsto \, Y(v,x)=\sum_{n \in \mathbb{Z}}v_{n}x^{-n-1}.
\end{align*}
We call $Y(v,x)$ the {\it vertex operator associated with} $v$.  We
assume that
\begin{align*}
Y(u,x)v \in V((x))
\end{align*}
for all $u,v \in V$, the {\it truncation property}.  Finally, we
require that the {\it Jacobi identity} is satisfied:
\begin{align}
\label{Jac}
x_{0}^{-1}\delta\left(\frac{x_{1}-x_{2}}{x_{0}}\right)Y(u,x_{1})Y(v,x_{2})&-
x_{0}^{-1}
\delta\left(\frac{-x_{2}+x_{1}}{x_{0}}\right)Y(v,x_{2})Y(u,x_{1}) \nonumber\\
&=x_{1}^{-1}\delta\left(\frac{x_{2}+x_{0}}{x_{1}}\right)Y(Y(u,x_{0})v,x_{2}).
\end{align}
\end{defi}

\begin{defi} \rm
An {\it injective ertex algebra} is an ertex algebra whose vertex
operator map is injective.
\end{defi}

\begin{remark} \rm
Following \cite{R} we referred to ``vacuum-free vertex algebras'' in
our introduction.  Vacuum-free vertex algebras are exactly the same as
what we have just called injective ertex algebras, but in the context
of this paper using the term injective ertex algebra seems clearer,
so from this point on, we shall not use the term ``vacuum-free vertex
algebra.''
\end{remark}

We may sometimes call an ertex algebra simply by the name of the
underlying vector space, or, to be more precise, sometimes by the
pair $(E,Y)$.

We note the following fact, although we shall only use it in the case
of an injective ertex algebra.
\begin{prop}
\label{prop:vfss}
\it Let $V$ be an ertex algebra.  For all $u, v \in V$, we have
\begin{align*}
Y(Y(u,x_{0})v,x_{2})=Y(Y(v, -x_{0})u,x_{2}+x_{0}).
\end{align*}
\end{prop} 
\begin{proof}
The proof is essentially a piece of a well-known proof of
skew-symmetry (cf. Proposition 3.1.19 in \cite{LL}).
\end{proof}
The following definition of ``vertex algebra'' is equivalent to the
definition in \cite{B} (cf. Proposition 3.6.6 in \cite{LL}):
\begin{defi} \rm
A \it{vertex algebra} \rm is an ertex algebra $(V,Y)$ together
with a distinguished element $\textbf{1}$ satisfying the 
\it{vacuum property} \rm
\begin{align*}
Y(\textbf{1},x)=1
\end{align*}
and the \it{creation property}\rm
\begin{align*}
Y(u,x)\textbf{1} \in V[[x]] \text{ and}\\
Y(u,0)\textbf{1}=u  \text{ for all }u \in V.
\end{align*}
\end{defi}
Every vertex algebra $V$ has a linear operator given by
\begin{align*}
\mathcal{D}(v)=v_{-2}{\bf 1} \qquad \text{for all} \quad v \in V.
\end{align*} 
We shall call this the ``derivative operator.''  

We may sometimes refer to a vertex algebra simply by the name of the
underlying vector space, but for more precision we shall also
sometimes refer to it using a triple such as $(V,Y,{\bf 1})$.

\begin{remark} \rm
The injectivity of the vertex operator map follows from the creation property.
\end{remark}
As we observed in \cite{R}, we have:
\begin{theorem} \rm
\label{main}
An injective ertex algebra can always be embedded into a vertex algebra. 
\end{theorem}

\begin{proof}
 The adjoint representation of any given injective ertex algebra $E$
yields a set of mutually local vertex operators, and by Theorem 3.2.10
in \cite{L} (cf. also Theorem 5.5.18 in \cite{LL}), these operators
generate a vertex algebra which we call $E''_{r}$.  The injectivity
shows that $E$ is embedded in $E''_{r}$.
\end{proof}

We now address the situation without using representation theory.  It
is tempting to simply try to directly adjoin a vacuum vector
$\textbf{1}$ to an ertex algebra $E$ to get a vertex algebra $V$, but
if this were possible we would necessarily have the {\it strong
creation property} (cf. (3.1.29) of \cite{LL}), which says that for
all $v \in V$,
\begin{align}
\label{strcr}
Y(v,x)\textbf{1}=e^{x\mathcal{D}}v.
\end{align}
Therefore, before adjoining a vacuum vector we might first close up
$E$ under some sort of linear map that will become $\mathcal{D}$.  Of
course, after closing $E$ under such an operator we may already have a
vertex algebra.

As a preliminary result let us consider the case in which we have an
ertex algebra which already has a linear map which has certain of the
properties of a derivative operator.

\begin{defi}
\label{def:Dert}
A $D$-ertex algebra is an ertex algebra $(E',Y')$ with a linear map 
\begin{align*}
D:E'
\rightarrow E'
\end{align*}
such that
\begin{align}
\label{D-Bracketexp1}
Y'(e^{zD}u,x)v=Y'(u,x+z)v,
\end{align}
\begin{align}
\label{D-Bracketexp2}
e^{zD}Y'(u,x)v=Y'(u,x+z)e^{zD}v,
\end{align}
and
\begin{align}
\label{skewsymm}
Y'(u,x)v=e^{xD}Y'(v,-x)u,
\end{align}
for all $u,v \in E'$.
\end{defi}
We may sometimes refer to a $D$-ertex algebra simply by the name of
the underlying vector space, but for more precision we shall also
sometimes refer to it using a triple such as $(E',Y',D)$.
\begin{remark} \rm
\label{rem:Dproperties}
The properties involving $D$ in Definition \ref{def:Dert}
are all well-known properties of any vertex algebra.  The first one
(\ref{D-Bracketexp1}) is the global form of the {\it
$\mathcal{D}$-derivative property} (cf. (3.1.28) in \cite{LL}), the
second one (\ref{D-Bracketexp2}) is the global form of the {\it
$\mathcal{D}$-bracket derivative formula} (cf. (3.1.35) in \cite{LL}),
and the third one (\ref{skewsymm}) is {\it skew-symmetry}
(cf. (3.1.30) in \cite{LL}).  We note that (\ref{D-Bracketexp1}) and
(\ref{D-Bracketexp2}) are equivalent to (any two of) three
``infinitesimal'' properties called the $\mathcal{D}$-derivative
property (cf. (3.1.25) in \cite{LL}) and $\mathcal{D}$-bracket
derivative formulas (cf. (3.1.32) and (3.1.33) in \cite{LL}), which
together say that $[\mathcal{D},Y(v,x)], \frac{d}{dx}Y(v,x)$ and
$Y(\mathcal{D}v,x)$ are all equal.  We have used two well-known global
forms encoding the same information in the official definition because
they seem convenient in our proofs.
\end{remark}
\begin{remark} \rm
\label{rem:Deq}
In Definition \ref{def:Dert} we only need
(\ref{skewsymm}) together with either one of (\ref{D-Bracketexp1}) or
(\ref{D-Bracketexp2}).  To show this we note that
\begin{align*}
Y'(u,x)e^{-yD}v&=e^{xD}Y'(e^{-yD}v,-x)u\\
&=e^{xD}Y'(v,-x-y)u\\
&=e^{xD}e^{(-x-y)D}Y'(u,x+y)v\\
&=e^{-yD}Y'(u,x+y)v,
\end{align*}
where the first and third equality follow from (\ref{skewsymm}) and
the second one from (\ref{D-Bracketexp1}).  This shows that we may
remove (\ref{D-Bracketexp2}) from Definition \ref{def:Dert}, and
cyclically switching the order of this calculation shows that we may
alternatively remove (\ref{D-Bracketexp1}) from Definition
\ref{def:Dert}.
\end{remark}

\begin{prop}
\label{prop:Dto1}
Given any $D$-ertex algebra there is a canonical embedding of it into
a vertex algebra which has codimension one.
\end{prop}

\begin{proof}
Let $(E',Y',D)$ be a $D$-ertex algebra and let
\begin{align*}
E''=E'\oplus C\textbf{1}, 
\end{align*}
where $\textbf{1}$ will become the
vacuum vector.

Let 
\begin{align*}
Y''(\cdot,x):E'' \otimes E'' \rightarrow E''[[x,x^{-1}]]
\end{align*}
be the unique linear map that extends $Y'(\cdot,x)$ on $E' \otimes E'$
and that satisfies
\begin{align}
\label{vac1}
Y''(u,x)\textbf{1}=e^{xD}u
\qquad \text{for all} \quad u \in E',
\end{align}
(where we identify $E'$ as a subspace of $E''$) and
\begin{align}
\label{vac2}
Y''(\textbf{1},x)u=u \qquad \text{for all} \quad u \in E''.
\end{align}

It is clear that all we need to show is that $E''$ together with $Y''$
is a vertex algebra.  We first show that $Y''$ is an ertex algebra.
By linearity, (\ref{vac1}) and (\ref{vac2}), together with the
truncation property on $E'$, we get the truncation property.

By linearity and since $E'$ is an ertex algebra, in order to check the
Jacobi identity, we only have to check the case where one or more of
the three vectors is ${\bf 1}$.  If $u={\bf 1}$ in (\ref{Jac}) then
the result follows from (\ref{vac2}) and (\ref{threetermdelta}).  If
$v={\bf 1}$ then by (\ref{vac1}) and (\ref{vac2}) the Jacobi identity
reduces to:
\begin{align*}
x_{0}^{-1}\delta\left(\frac{x_{1}-x_{2}}{x_{0}}\right)Y''(u,x_{1})&-
x_{0}^{-1}
\delta\left(\frac{-x_{2}+x_{1}}{x_{0}}\right)Y''(u,x_{1})\\ 
&=x_{1}^{-1}\delta\left(\frac{x_{2}+x_{0}}{x_{1}}\right)Y''(e^{x_{0}D}u,x_{2}),
\end{align*}
which in turn, by (\ref{twotermdelta}), reduces to
\begin{align*}
x_{1}^{-1}\delta\left(\frac{x_{2}+x_{0}}{x_{1}}\right)Y''(u,x_{1})
&=x_{1}^{-1}\delta\left(\frac{x_{2}+x_{0}}{x_{1}}\right)Y''(e^{x_{0}D}u,x_{2}),
\end{align*}
which, assuming we are not acting against ${\bf 1}$, follows by
Proposition \ref{prop:deltasub} and (\ref{D-Bracketexp1}).  And if we
are acting against ${\bf 1}$ then by (\ref{vac1}) we need
\begin{align*}
x_{1}^{-1}\delta\left(\frac{x_{2}+x_{0}}{x_{1}}\right)e^{x_{1}D}u
&=x_{1}^{-1}\delta\left(\frac{x_{2}+x_{0}}{x_{1}}\right)e^{(x_{2}+x_{0})D}u,
\end{align*}
which follows by Proposition \ref{prop:deltasub}.

Therefore for $u,v \in E'$ identified as elements of $E''$, we need
only check
\begin{align*}
x_{0}^{-1}\delta\left(\frac{x_{1}-x_{2}}{x_{0}}\right)
Y''(u,x_{1})Y''(v,x_{2}){\bf 1}&-
x_{0}^{-1}
\delta\left(\frac{-x_{2}+x_{1}}{x_{0}}\right)
Y''(v,x_{2})Y''(u,x_{1}){\bf 1} \nonumber\\
&=x_{1}^{-1}\delta\left(\frac{x_{2}+x_{0}}{x_{1}}\right)
Y''(Y''(u,x_{0})v,x_{2}){\bf 1},
\end{align*}
which, by (\ref{vac1}), reduces to 
\begin{align*}
x_{0}^{-1}\delta\left(\frac{x_{1}-x_{2}}{x_{0}}\right)
Y'(u,x_{1})e^{x_{2}D}v&-
x_{0}^{-1}
\delta\left(\frac{-x_{2}+x_{1}}{x_{0}}\right)
Y'(v,x_{2})e^{x_{1}D}u \nonumber\\
&=x_{1}^{-1}\delta\left(\frac{x_{2}+x_{0}}{x_{1}}\right)
e^{x_{2}D}Y'(u,x_{0})v,
\end{align*}
which, by Proposition \ref{prop:deltasub} and (\ref{D-Bracketexp2}), 
reduces to
\begin{align*}
x_{0}^{-1}\delta\left(\frac{x_{1}-x_{2}}{x_{0}}\right)
e^{x_{2}D}Y'(u,x_{0})v
&-x_{1}^{-1}\delta\left(\frac{x_{2}+x_{0}}{x_{1}}\right)
e^{x_{2}D}Y'(u,x_{0})v\\
&=
x_{0}^{-1}
\delta\left(\frac{-x_{2}+x_{1}}{x_{0}}\right)
Y'(v,x_{2})e^{x_{1}D}u,
\end{align*}
which by (\ref{twotermdelta}) (checking carefully that the left hand
expression is well-defined in the sense that all coefficients of
monomials are finitely computable) reduces to 
\begin{align*}
x_{0}^{-1}
\delta\left(\frac{-x_{2}+x_{1}}{x_{0}}\right)
e^{x_{2}D}Y'(u,x_{0})v
=
x_{0}^{-1}
\delta\left(\frac{-x_{2}+x_{1}}{x_{0}}\right)
Y'(v,x_{2})e^{x_{1}D}u,
\end{align*}
which by Proposition \ref{prop:deltasub} and taking
$\text{Res}_{x_{0}}$ reduces to
\begin{align*}
e^{x_{2}D}Y'(u,-x_{2}+x_{1})v
=
Y'(v,x_{2})e^{x_{1}D}u,
\end{align*}
which by (\ref{D-Bracketexp1}) reduces to 
\begin{align*}
e^{x_{2}D}Y'(e^{x_{1}D}u,-x_{2})v
=
Y'(v,x_{2})e^{x_{1}D}u,
\end{align*}
which follows from (\ref{skewsymm}).

The vacuum property and creation property both follow immediately from
(\ref{vac1}) and (\ref{vac2}).
\end{proof}
\begin{remark} \rm
We note, as in the case of a rng being canonically embedded into a
ring, that even if $E'$ is already a vertex algebra, we still add
a new vacuum vector so that $E'$ indeed always has codimension $1$.
\end{remark}
\begin{remark} \rm
It is easy to see by (\ref{vac1}) that the map $D$ on $E'$ extends to
the derivative operator on $E''$.
\end{remark}

We would like to be able to embed ertex algebras into $D$-ertex
algebras.  Of course, we have a non-canonical method using the
representation theoretic proof of Theorem \ref{main}.  We could also
embed an injective ertex algebra more directly using a representation
theoretic proof by closing up the weak vertex operator space under
derivatives instead of generating a larger space by adding the
identity.  We shall do this below for interest's sake, but we note
that it is unclear whether or not the resulting $D$-ertex algebra is
injective.  As noted below in Remark \ref{Drepnotcanon} this makes
fail a naive attempt at proving whether or not this construction is
canonical.  I leave to the reader the question of whether or not it is
canonical.  In the next section we give a canonical embedding of an
injective ertex algebra into a $D$-ertex algebra using a proof which
does not rely on representation theory and one of the key properties
of the construction is that the $D$-ertex algebra of this construction
is, in fact, injective.

\begin{corollary}
\label{cor:main}
Every injective ertex algebra may be embedded into a $D$-ertex algebra.
\end{corollary}

\begin{proof}
The result follows immediately from Theorem \ref{main}, but we give a
second proof.

Let $E$ be an injective ertex algebra.  Consider its adjoint
representation which is faithful because of the injectivity of $E$.
Take the linear span of all the higher derivatives of the weak vertex
operators in the adjoint representation and call this space $E'_{r}$.
We claim that this vector space is a $D$-ertex algebra of weak vertex
operators.  Of course, if the claim is true then $E'_{r}$ is the
smallest $D$-ertex algebra of weak vertex operators containing the
adjoint image of $E$ which is in $(\mathcal{E}(E), Y_{\mathcal{E}},
1_{E})$, which by Proposition 5.3.9 of \cite{LL} is a weak vertex
algebra.  By smallest we mean that it is contained in all other such
$D$-ertex algebras or in other words that it is the $D$-ertex algebra
generated by the adjoint image of $E$.

To show this claim we first show that if $a(x)$ and $b(x)$ are
mutually local then $a'(x)$ and $b(x)$ are mutually local.  There
exists some $k \geq 0$ so that
\begin{align*}
(x_{1}-x_{2})^{k}a(x_{1})b(x_{2})=(x_{1}-x_{2})^{k}b(x_{2})a(x_{1})
\end{align*}
and taking $\frac{\partial}{\partial x_{1}}$ of both sides gives
\begin{align*}
k(x_{1}-x_{2})^{k-1}a(x_{1})b(x_{2})&+
(x_{1}-x_{2})^{k}a'(x_{1})b(x_{2})\\
&=
k(x_{1}-x_{2})^{k-1}b(x_{2})a(x_{1})+
(x_{1}-x_{2})^{k}b(x_{2})a'(x_{1}),
\end{align*}
so that
\begin{align*}
(x_{1}-x_{2})^{k+1}a'(x_{1})b(x_{2})=(x_{1}-x_{2})^{k+1}b(x_{2})a'(x_{1}).
\end{align*}
Thus by induction $E'_{r}$ is a mutually local subspace of
$(\mathcal{E}(E), Y_{\mathcal{E}}, 1_{E})$.  Therefore $E'_{r}$ is
embedded in a vertex subalgebra of $(\mathcal{E}(E), Y_{\mathcal{E}},
1_{E})$ by Theorem 3.2.10 in \cite{L} (cf. also Theorem 5.5.18 in
\cite{LL}).  Moreover the $\mathcal{D}$-bracket and derivative
properties of $(\mathcal{E}(E), Y_{\mathcal{E}}, 1_{E})$ show, by
induction, that $E'_{r}$ is closed under multiplication.  Since
$E'_{r}$ is also obviously closed under the derivative operator then
$E'_{r}$ is a $D$-ertex algebra.

Finally, since $E$ is injective therefore $E$ is embedded in $E'_{r}$.
\end{proof}

\section{Main result}
\label{sec:main}
Let $(E,Y)$ be an injective ertex algebra.  We shall embed $(E,Y)$
into an injective $D$-ertex algebra $(E',Y',D)$.  Our construction
will be basis dependent in this section, but in Theorem \ref{U2} we
give a basis-free characterization of $(E',Y',D)$.  Proposition
\ref{prop:Dto1} motivates our approach.  Then by invoking Proposition
\ref{prop:Dto1}, we canonically embed $E'$ and hence $E$ into a vertex
algebra.

Let $\{e_{i}\}$, where $i$ ranges over an index set $I$, be a basis of
$E$.  Let 
\begin{align*}
S=\{D^{[n]}e_{i}|i \in I, n\geq 0\}.  
\end{align*}
We emphasize that the symbols $D^{[n]}e_{i}$ are simply the names of a
doubly-indexed set and that, though intentionally suggestive, the
notation does not denote operators acting on a set.  For convenience
we shall abuse notation and sometimes identify $D^{[0]}e_{i}$ with
$e_{i}$.  Let $\bar{S}$ be the complex vector space with basis $S$.

\begin{defi} \rm
Let $\bar{Y}(\cdot,x)\cdot:\bar{S} \otimes E \rightarrow
E[[x,x^{-1}]]$ be the unique linear map satisfying
\begin{align*}
\bar{Y}(D^{[n]}e_{i},x)e_{j}=\left(\frac{d}{dx}\right)^{n}Y(e_{i},x)e_{j},
\end{align*}
for all $n \geq 0$ $i,j \in I$.
\end{defi}

Let $A$ be the collection of subsets of $S$ containing
$\{D^{[0]}e_{i}|i \in I\}$ satisfying the following property: if $R
\in A$ and if $m$ is a nonzero linear combination of elements of $R$,
regarding $m$ as an element of $\bar{S}$, we require that
$\bar{Y}(m,x)e=0$ for all $e \in E$.  Since $E$ is injective, $A$ is
nonempty.  Let $T$ be a simply ordered (by set inclusion)
sub-collection of $A$.  It is easy to see that the union of elements
of elements of $T$ is still an element of $A$.  Therefore, by Zorn's
lemma, $A$ has a maximal element.  Take such a maximal element $M$ and
let its elements give a basis for a vector space that we call $E'$.
\begin{lemma}
If $D^{[n]}e_{i} \notin M$ then there exists a unique element in $E'$
which we shall call $\bar{D}^{[n]}e_{i}$, such that
\begin{align*}
\bar{Y}(D^{[n]}e_{i}-\bar{D}^{[n]}e_{i},x)=0
\end{align*}
when acting on $E$. 
\end{lemma}
\begin{proof}
The existence follows from the maximality of $M$ and the uniqueness
follows from the restricting condition on all elements of $A$.
\end{proof}
\begin{defi}
Let $D$ be the unique linear operator on $E'$ which satisfies
\begin{align*}
   D(D^{[n]}e_{i}) = \left\{
     \begin{array}{lr}
       D^{[n+1]}e_{i} &  D^{[n+1]}e_{i} \in M\\
       \bar{D}^{[n+1]}e_{i} & \text{otherwise}
     \end{array}
   \right.
\end{align*}
for all $D^{[n]}e_{i} \in M$.
\end{defi}

\begin{defi}
Let $Y'(\cdot,x)\cdot :E' \otimes E' \rightarrow E'[[x,x^{-1}]]$ be
the unique linear map which satisfies
\begin{align}
\label{eq:Y'}
Y'(D^{[n]}e_{i},x)D^{[m]}e_{j}=
\left(\frac{d}{dx}\right)^{n}\left(D-\frac{d}{dx}\right)^{m}Y(e_{i},x)e_{j},
\end{align}
whenever $D^{[n]}e_{i}$ and $D^{[m]}e_{j} \in M$.
\end{defi}
We note that (\ref{eq:Y'}) may be rewritten as
\begin{align}
Y'(D^{[n]}e_{i},x)D^{[m]}e_{j}=
\text{Res}_{z}\text{Res}_{y}
z^{-n-1}y^{-m-1}
e^{z\frac{d}{dx}+y\left(D-\frac{d}{dx}\right)}
Y(e_{i},x)e_{j} \nonumber\\
=
\text{Res}_{z}\text{Res}_{y}
z^{-n-1}y^{-m-1}
e^{yD}
Y(e_{i},x+z-y)e_{j}. \label{genY'}
\end{align}
\begin{lemma}
\label{lem:EtoE'}
If $Y'(u,x)=Y'(v,x)$ where $u,v \in E'$ when acting on $E$ then
$Y'(u,x)=Y'(v,x)$ (when acting on $E'$).
\end{lemma}
\begin{proof}
We have
\begin{align*}
Y'(u,x)D^{[m]}e_{j}&=\left(D-\frac{d}{dx}\right)^{m}Y'(u,x)e_{j}\\
&=\left(D-\frac{d}{dx}\right)^{m}Y'(v,x)e_{j}\\
&=Y'(v,x)D^{[m]}e_{j},
\end{align*}
for all $D^{[m]}e_{j} \in M$ so that the result follows by linearity.
\end{proof}
\begin{remark} \rm
\label{rem:EtoE'}
Of course, a similar argument shows that if
$Y'(u,x)=\frac{d}{dx}Y'(v,x)$ where $u,v \in E'$ when acting on $E$
then $Y'(u,x)=\frac{d}{dx}Y'(v,x)$ (when acting on $E'$).
\end{remark}
\begin{lemma}
\label{lem:Y'inj}
If $Y'(u,x)=0$ then $u=0$ for all $u \in E'$.
\end{lemma}
\begin{proof}
When acting on $E$ we have $Y'(u,x)=\bar{Y}(u,x)$.  Therefore the
result follows by the restrictive condition in the definition of $A$.
\end{proof}
\begin{lemma}
Let $u,v \in E'$.  Then $\bar{Y}(Du,x)=\frac{d}{dx}\bar{Y}(u,x)$ (when
acting on $E$).
\end{lemma}
\begin{proof}
Consider $D^{[n]}e_{i} \in E'$.  If $D(D^{[n]}e_{i})=D^{[n+1]}e_{i}$
then we have
\begin{align*}
\bar{Y}(D(D^{[n]}e_{i}),x)
=\bar{Y}(D^{[n+1]}e_{i},x)
=\frac{d}{dx}\bar{Y}(D^{[n]}e_{i},x).
\end{align*}
Otherwise $D(D^{[n]}e_{i})=\bar{D}^{[n+1]}e_{i}$ and we have
\begin{align*}
\bar{Y}(D(D^{[n]}e_{i}),x)
=\bar{Y}(\bar{D}^{[n+1]}e_{i},x)
=\bar{Y}(D^{[n+1]}e_{i},x)
=\frac{d}{dx}\bar{Y}(D^{[n]}e_{i},x),
\end{align*}
so the result follows by linearity.
\end{proof}
\begin{lemma}
\label{lem:dder}
For all $u\in E'$, $Y'(Du,x)=\frac{d}{dx}Y'(u,x)$ (when
acting on $E'$).
\end{lemma}
\begin{proof}
Consider $D^{[n]}e_{i} \in E'$.  If $D(D^{[n]}e_{i})=D^{[n+1]}e_{i}$, we have
\begin{align*}
Y'(D(D^{[n]}e_{i}),x)
=Y'(D^{[n+1]}e_{i},x)
=\frac{d}{dx}Y'(D^{[n]}e_{i},x).
\end{align*}
Otherwise $D(D^{[n]}e_{i})=\bar{D}^{[n+1]}e_{i}$ and we have
\begin{align*}
Y'(D(D^{[n]}e_{i}),x)
=Y'(\bar{D}^{[n+1]}e_{i},x)
=\bar{Y}(\bar{D}^{[n+1]}e_{i},x)
=\frac{d}{dx}\bar{Y}(D^{[n]}e_{i},x)
=\frac{d}{dx}Y'(D^{[n]}e_{i},x)
\end{align*}
where some of these equalities hold only on when acting $E$.  Then the
result follows by Lemma \ref{lem:EtoE'} together with Remark
\ref{rem:EtoE'} and linearity.
\end{proof}
Lemma \ref{lem:dder}, together with the formal Taylor theorem,
gives for all $u \in E'$
\begin{align}
\label{eq:globdder}
Y'(e^{zD}u,x)=Y'(u,x+z).
\end{align}
\begin{lemma}
\label{lem:sksy}
For all $u',v' \in E'$ we have $Y'(u',x)v'=e^{xD}Y'(v',-x)u'$.  
\end{lemma}
\begin{proof}
We have for all $u,v \in E$ when acting on $E$:
%\begin{align*}
%=x_{1}^{-1}\delta \left(\frac{x_{2}+x_{0}}{x_{1}}\right)
%Y'(Y(u,x_{0})v,x_{2})
%=x_{2}^{-1}\delta \left(\frac{x_{1}-x_{0}}{x_{2}}\right)
%Y(Y(u,x_{0})v,x_{2})\\
%=x_{1}^{-1}\delta \left(\frac{x_{2}+x_{0}}{x_{1}}\right)
%Y(Y(v,-x_{0})u,x_{1})\\
%=x_{1}^{-1}\delta \left(\frac{x_{2}+x_{0}}{x_{1}}\right)
%Y(Y(v,-x_{0})u,x_{2}+x_{0})\\
%=x_{1}^{-1}\delta \left(\frac{x_{2}+x_{0}}{x_{1}}\right)
%e^{x_{0}\frac{d}{dx_{2}}}
%Y(Y(v,-x_{0})u,x_{2})\\
%=x_{1}^{-1}\delta \left(\frac{x_{2}+x_{0}}{x_{1}}\right)
%e^{x_{0}\frac{d}{dx_{2}}}
%Y'(Y(v,-x_{0})u,x_{2})\\
%=x_{1}^{-1}\delta \left(\frac{x_{2}+x_{0}}{x_{1}}\right)
%Y'(e^{x_{0}D}Y(v,-x_{0})u,x_{2})
%\end{align*}
\begin{align*}
Y'(Y(u,x_{0})v,x_{2})
&=Y(Y(u,x_{0})v,x_{2})\\
&=Y(Y(v,-x_{0})u,x_{2}+x_{0})\\
&=e^{x_{0}\frac{d}{dx_{2}}}
Y(Y(v,-x_{0})u,x_{2})\\
&=e^{x_{0}\frac{d}{dx_{2}}}
Y'(Y(v,-x_{0})u,x_{2})\\
&=Y'(e^{x_{0}D}Y(v,-x_{0})u,x_{2})
\end{align*}
where the second equality follows from Proposition \ref{prop:vfss},
the third equality follows from the formal Taylor theorem and the fifth
equality follows from Lemma \ref{lem:dder}.  Then Lemma
\ref{lem:EtoE'} and Lemma \ref{lem:Y'inj} give
\begin{align*}
Y'(u,x)v=e^{xD}Y'(v,-x)u.
\end{align*}
Further recalling (\ref{genY'}) we get
\begin{align*}
Y'(D^{[n]}e_{i},x)D^{[m]}e_{j}
&=
\text{Res}_{z}\text{Res}_{y}
z^{-n-1}y^{-m-1}
e^{yD}
Y'(e_{i},x+z-y)e_{j}\\
&=\text{Res}_{z}\text{Res}_{y}
z^{-n-1}y^{-m-1}
e^{yD}e^{(x+z-y)D}
Y'(e_{j},-x-z+y)e_{i}\\
&=\text{Res}_{z}\text{Res}_{y}
z^{-n-1}y^{-m-1}
e^{(x+z)D}
Y'(e_{j},-x-z+y)e_{i}\\
&=e^{xD}
\text{Res}_{y}\text{Res}_{z}
y^{-m-1}z^{-n-1}e^{zD}Y'(e_{j},-x+y-z)e_{i}\\
&=e^{xD}Y'(D^{[m]}e_{j},-x)D^{[n]}e_{i},
\end{align*}
so that the result follows by linearity.
\end{proof}
%\begin{lemma}
%Let $u,v \in E'$.  Then $[D,Y'(u,x)]v=\frac{d}{dx}Y'(u,x)v$.
%\end{lemma}
%\begin{proof}
%\begin{align}
%\frac{d}{dx}Y(u,x)v&=\mathcal{D}e^{x\mathcal{D}}Y(v,-x)u
%+e^{x\mathcal{D}}\frac{d}{dx}Y(v,-x)u \nonumber\\
%&=\mathcal{D}Y(u,x)v+e^{x\mathcal{D}}\frac{d}{dx}Y(v,-x)u \nonumber\\
%&=\mathcal{D}Y(u,x)v-e^{x\mathcal{D}}Y(\mathcal{D}v,-x)u \nonumber \\
%&=\mathcal{D}Y(u,x)v-Y(u,x)\mathcal{D}v,
%\end{align}
%where...
%\end{proof}
\begin{lemma}
\label{lem:globdbrac}
For all $u,v \in E'$ we have $e^{yD}Y'(u,x)e^{-yD}v=Y'(u,x+y)v$.
\end{lemma}
\begin{proof}
See Remark \ref{rem:Deq}.
%We have:
%\begin{align*}
%Y'(u,x)e^{-yD}v&=e^{xD}Y'(e^{-yD}v,-x)u\\
%&=e^{xD}Y'(v,-x-y)u\\
%&=e^{xD}e^{(-x-y)D}Y'(u,x+y)v\\
%&=e^{-yD}Y'(u,x+y)v.
%\end{align*}
\end{proof}

We have now shown that $E'$ has all of the properties specified in
Proposition \ref{prop:Dto1}, that is, assuming that $E'$ is indeed an
ertex algebra.  Showing that is then essentially our last task because
then we will have that $(E,Y)$ is obviously embedded in $(E',Y')$ and
can invoke Proposition \ref{prop:Dto1}.

%Then take the abstract vector space with basis
%$\mathcal{D}^{(n)}e_{i}$ $n \geq 0, i \in I$ (these are just the names
%of vectors at this stage).  Call this space $E'$.  We shall give $E'$
%a natural ertex algebra structure.  First we define $\mathcal{D}$ to
%be the unique linear operator on $E'$ which satisfies
%\begin{align*}
%\mathcal{D}(\mathcal{D}^{(n)}e_{i})=\mathcal{D}^{(n+1)}e_{i} 
%\qquad \qquad \text{for all} \quad n \geq 0, i \in I.
%\end{align*}

%These properties taken together motivate the following definition.
%\begin{defi}
%let $Y'(\cdot,x):E' \otimes E' \rightarrow E'[[x,x^{-1}]]$ be the
% unique linear map that satisfies
%\begin{align}
%\label{eq:D'}
%Y'(e^{y\mathcal{D}}u,x)e^{z\mathcal{D}}v=e^{zD}Y(u,x+y-z)v
%\qquad \text{for all} \quad u,v \in E,
%\end{align}
%where we identify $E$ as a subspace of $E'$.
%\end{defi}
%Clearly, by formally setting $y=z=0$ in (\ref{eq:D'}) we see that
%$Y'(\cdot,x)$ restricts to $Y(\cdot,x)$.
\begin{theorem} \rm
\label{lem:E'ert}
The vector space $E'$ together with map $Y'(\cdot,x)$ is an ertex algebra.
\end{theorem}
\begin{proof}
We first show that $Y'(\cdot,x)\cdot$ satisfies the truncation
property.  We have for $u,v \in E$
\begin{align} 
\label{eqfortrunc}
Y'(e^{yD}u,x)e^{zD}v
=e^{zD}Y(u,x+y-z)v
=e^{zD}e^{(y-z)\frac{\partial}{\partial x}}Y(u,x)v,
\end{align}
where the first equality follows from (\ref{eq:globdder}) and Lemma
\ref{lem:globdbrac} and the second equality follows from the formal
Taylor theorem.  Therefore, if we extract the coefficient for a fixed
power of both $y$ and $z$ then only finitely many terms from the
exponential terms $e^{zD}$ and $e^{(y-z)\frac{\partial}{\partial
x}}$need to be considered.  We have that $Y(u,x)v$ is truncated from
below in powers of $x$.  Thus the coefficient of any fixed powers of
both $y$ and $z$ of (\ref{eqfortrunc}), since we must take at most a
bounded number of derivatives, we find is also truncated from below in
powers of $x$.  The truncation property now follows from linearity.

What remains is to verify the Jacobi identity.  We have that for all $u,v$
and $w \in E$
\begin{align*}
x_{0}^{-1}\delta\left(\frac{x_{1}-x_{2}}{x_{0}}\right)
Y'(e^{yD}u,x_{1})&Y'(e^{zD}v,x_{2})e^{tD}w\\
-
x_{0}^{-1}\delta\left(\frac{-x_{2}+x_{1}}{x_{0}}\right)
&Y'(e^{zD}v,x_{2})Y'(e^{yD}u,x_{1})e^{tD}w\\
&
=x_{1}^{-1}\delta\left(\frac{x_{2}+x_{0}}{x_{1}}\right)
Y'(Y'(e^{yD}u,x_{0})e^{zD}v,x_{2})e^{tD}w
\end{align*}
is equivalent to
\begin{align*}
x_{0}^{-1}\delta\left(\frac{x_{1}-x_{2}}{x_{0}}\right)
Y'(e^{yD}u,x_{1})&e^{tD}Y(v,x_{2}+z-t)w\\
-
x_{0}^{-1}\delta\left(\frac{-x_{2}+x_{1}}{x_{0}}\right)
&Y'(e^{zD}v,x_{2})e^{tD}Y(u,x_{1}+y-t)w\\
&=
x_{1}^{-1}\delta\left(\frac{x_{2}+x_{0}}{x_{1}}\right)
Y'(e^{zD}Y(u,x_{0}+y-z)v,x_{2})e^{tD}w
\end{align*}
which is the same as 
\begin{align*}
x_{0}^{-1}\delta\left(\frac{x_{1}-x_{2}}{x_{0}}\right)
e^{tD}Y(u,x_{1}+y-t)Y(v,x_{2}+z-t)w&\\
-
x_{0}^{-1}\delta\left(\frac{-x_{2}+x_{1}}{x_{0}}\right)
e^{tD}Y(v,x_{2}+z-t)&Y(u,x_{1}+y-t)w\\
=x_{1}^{-1}\delta\left(\frac{x_{2}+x_{0}}{x_{1}}\right)
e^{tD}&Y(Y(u,x_{0}+y-z)v,x_{2}+z-t)w
\end{align*}
which is equivalent to
\begin{align*}
x_{0}^{-1}\delta\left(\frac{x_{1}-x_{2}}{x_{0}}\right)
e^{tD}
e^{(y-t)\frac{\partial}{\partial x_{1}}+(z-t)\frac{\partial}{\partial x_{2}}}
Y(u,x_{1})&Y(v,x_{2})w\\
-
x_{0}^{-1}\delta\left(\frac{-x_{2}+x_{1}}{x_{0}}\right)
e^{tD}
&e^{(y-t)\frac{\partial}{\partial x_{1}}+(z-t)\frac{\partial}{\partial x_{2}}}
Y(v,x_{2})Y(u,x_{1})w\\
=x_{1}^{-1}\delta\left(\frac{x_{2}+x_{0}}{x_{1}}\right)
&e^{tD}
e^{(y-z)\frac{\partial}{\partial x_{0}}+(z-t)\frac{\partial}{\partial x_{2}}}
Y(Y(u,x_{0})v,x_{2})w
\end{align*}
which is the same as 
\begin{align*}
e^{tD}
e^{(y-t)\frac{\partial}{\partial x_{1}}+(z-t)\frac{\partial}{\partial x_{2}}}
x_{0}^{-1}\delta\left(\frac{x_{1}-x_{2}+z-y}{x_{0}}\right)
Y(u,x_{1})&Y(v,x_{2})w\\
-
e^{tD}
e^{(y-t)\frac{\partial}{\partial x_{1}}+(z-t)\frac{\partial}{\partial x_{2}}}
x_{0}^{-1}\delta\left(\frac{-x_{2}+x_{1}+z-y}{x_{0}}\right)
&Y(v,x_{2})Y(u,x_{1})w\\
=
e^{tD}
e^{(y-z)\frac{\partial}{\partial x_{0}}+(z-t)\frac{\partial}{\partial x_{2}}}
x_{1}^{-1}\delta\left(\frac{x_{2}+x_{0}+t-y}{x_{1}}\right)
&Y(Y(u,x_{0})v,x_{2})w,
\end{align*}
which is equivalent to
\begin{align*}
e^{tD}
e^{(y-z)\frac{\partial}{\partial x_{0}}
+(y-t)\frac{\partial}{\partial x_{1}}+(z-t)\frac{\partial}{\partial x_{2}}}
x_{0}^{-1}\delta\left(\frac{x_{1}-x_{2}}{x_{0}}\right)
Y(u,x_{1})&Y(v,x_{2})w\\
-
e^{tD}
e^{(y-z)\frac{\partial}{\partial x_{0}}
+(y-t)\frac{\partial}{\partial x_{1}}+(z-t)\frac{\partial}{\partial x_{2}}}
x_{0}^{-1}\delta\left(\frac{-x_{2}+x_{1}}{x_{0}}\right)
&Y(v,x_{2})Y(u,x_{1})w\\
=
e^{tD}
e^{(y-z)\frac{\partial}{\partial x_{0}}
+(y-t)\frac{\partial}{\partial x_{1}}+(z-t)\frac{\partial}{\partial x_{2}}}
&x_{1}^{-1}\delta\left(\frac{x_{2}+x_{0}}{x_{1}}\right)
Y(Y(u,x_{0})v,x_{2})w,
\end{align*}
which follows from the Jacobi identity.  The first two equalities
follow from (\ref{eq:globdder}) and Lemma \ref{lem:globdbrac}, the
next two by the formal Taylor theorem, the next by
(\ref{twotermdelta}) together with the formal Taylor theorem and the
last because we may factor out the operator $e^{tD}
e^{(y-z)\frac{\partial}{\partial x_{0}} +(y-t)\frac{\partial}{\partial
x_{1}}+(z-t)\frac{\partial}{\partial x_{2}}}$ from all three terms and
use that $E$ is already an ertex algebra.  The result now follows by
linearity.
\end{proof}

We may now give a second proof of Theorem \ref{main}, which does not
rely on any representation theory.  

\begin{proof}
(Second proof of Theorem \ref{main}) The result follows from Lemmas
\ref{lem:E'ert}, \ref{lem:sksy} and \ref{lem:globdbrac},
(\ref{eq:globdder}) and Proposition \ref{prop:Dto1} so that $E$ is
embedded into $E'$ which in turn is embedded into $E''$.
\end{proof}
We give a basis free characterization of the $E''$ in the last proof
in Corollary \ref{U3}.
\section{The canonical properties}
We shall show that certain of our constructions are canonical.  We
shall define and recall the relevant notions of homomorphism needed to
state the properties of interest.

An ertex algebra homomorphism between two ertex algebras $(E_{1},
Y_{1})$ and $(E_{2},Y_{2})$ is a linear map $f:E_{1} \rightarrow
E_{2}$ such that
\begin{align*}
f(Y_{1}(u,x)v)=Y_{2}(f(u),x)f(v), \qquad \text{for all} \quad u,v \in E_{1}.
\end{align*}

This extends to the usual definition of a vertex algebra homomorphism,
which we recall now.  A vertex algebra homomorphism between vertex
algebras $(V_{1}, Y_{1}, \textbf{1}_{1})$ and $(V_{2},Y_{2},
\textbf{1}_{2})$ is a linear map $f:V_{1} \rightarrow V_{2}$ such
that
\begin{align*}
f(Y_{1}(u,x)v)=Y_{2}(f(u),x)f(v),
\end{align*}
and such that
\begin{align*}
f(\textbf{1}_{1})=\textbf{1}_{2},
\end{align*}
or in other words, a vertex algebra homomorphism is a homomorphism of
vertex algebras regarded as ertex algebras which, in addition,
preserves the vacuum vector.

A $D$-ertex algebra homomorphism between two
$D$-ertex algebras $(E_{1}, Y_{1},D_{1})$ and\\
$(E_{2},Y_{2},D_{2})$ is a linear map $f:E_{1} \rightarrow E_{2}$ such
that
\begin{align*}
f(Y_{1}(u,x)v)=Y_{2}(f(u),x)f(v), \qquad \text{for all} \quad u,v \in E_{1}.
\end{align*}
and such that
\begin{align*}
f(D_{1}v)=D_{2}f(v) \qquad \text{for all} \quad v \in E_{1},
\end{align*}
or in other words, a $D$-ertex algebra homomorphism is a
homomorphism of $D$-ertex algebras regarded as ertex
algebras which, in addition, is compatible with the respective
derivative operators.

\begin{prop}
Let $f$ be a vertex algebra homomorphism between vertex algebras\\
$(V_{1}, Y_{1}, \textbf{1}_{1})$ and $(V_{2},Y_{2}, \textbf{1}_{2})$.
Let $D_{1}(v)=v_{-2}\textbf{1}_{1}$ and
$D_{2}(v)=v_{-2}\textbf{1}_{2}$.  Then $f$ is a $D$-ertex
algebra homomorphism between $(V_{1}, Y_{1}, D_{1})$ and $(V_{2},
Y_{2}, D_{2})$.
\end{prop}
\begin{proof}
It is clear (see Remark \ref{rem:Dproperties}) that $(V_{1}, Y_{1},
D_{1})$ and $(V_{2}, Y_{2}, D_{2})$ are, in fact, $D$-ertex
algebras.  Since $f$ is a vertex algebra homomorphism we have
\begin{align*}
f(Y_{1}(v,x)\textbf{1}_{1})
=Y_{2}(f(v),x)f(\textbf{1}_{1})
=Y_{2}(f(v),x)\textbf{1}_{2},
\end{align*}
which by the strong creation property gives
\begin{align*}
f(e^{xD_{1}}v)=e^{xD_{2}}f(v)
\end{align*}
and equating the coefficients of the linear term yields
\begin{align*}
f(D_{1}v)=D_{2}f(v).
\end{align*}
\end{proof}
Given any vertex algebras when we regard them as $D$-ertex
algebras we always use the standard derivative operator.

The next proposition shows that the vertex algebra $E''$ from
Proposition \ref{prop:Dto1} satisfies a universal property.
\begin{prop}
\label{U1}
Using the notation of the statement and proof of Proposition
\ref{prop:Dto1} let $i$ be the canonical embedding of $E'$ into $E''$.
Given any vertex algebra $(V,Y_{V},{\bf 1}_{V})$ such that there
exists a $D$-ertex algebra homomorphism, $\psi:E' \rightarrow
V$, there exists a unique vertex algebra homomorphism $\phi: E''
\rightarrow V$ such that $\phi (i(u))=\psi(u)$ for all $u \in E'$, or
in other words such that the diagram
%\begin{align*}
%\begin{CD}
%E' @>i>> E''\\
%@VV \psi V   @VV \phi V \\
%V @= V.
%\end{CD}
%\end{align*}
\begin{align*}
\xymatrix{
E' \ar[rd]^\psi \ar[r]^i & E'' \ar[d]^\phi\\
&V}
\end{align*}
commutes.

Furthermore, $E''$ is the unique, up to canonical isomorphism, vertex
algebra equipped with an embedding of $E'$ into it which satisfies the
above property.
\end{prop}

\begin{proof}
It is obvious by linearity that if $\phi$ exists then it must be
unique.  Indeed any element of $E''$ may be written as $i(u)+a{\bf 1}$
where $u \in E'$ and $a \in \mathbb{C}$ and we must have
\begin{align*}
\phi(i(u)+a{\bf 1})=\psi(u)+a{\bf 1}_{V}.
\end{align*}
Then $\phi$ obviously preserves the vacuum vector and we need only
show that it is indeed a vertex algebra homomorphism.  We have for all
$u,v \in E'$ and $a,b \in \mathbb{C}$
\begin{align*}
\phi\left(Y''(i(u)+a{\bf 1},x)(i(v)+b{\bf 1})\right)&=
\phi\left(Y''(i(u),x)(i(v))\right)
+\phi\left(Y''(i(u),x)b{\bf 1}\right)\\
&\quad +\phi(ai(v)+ab{\bf 1})\\
&=\phi \circ i\left(Y'(u,x)v\right)
+\phi be^{xD}i(u)
+a\psi(v)+ab{\bf 1}_{V}\\
&=\psi \left(Y'(u,x)v\right)
+\psi be^{xD}u
+a\psi(v)+ab{\bf 1}_{V}\\
&=Y_{V}(\psi(u),x)\psi(v)
+be^{xD_{V}}\psi(u)
+a\psi(v)+ab{\bf 1}_{V}\\
&=Y_{V}(\psi(u),x)\psi(v)
+Y_{V}(\psi(u),x)b{\bf 1}_{V}
+a\psi(v)+ab{\bf 1}_{V}\\
&=Y_{V}(\psi(u)+a{\bf 1}_{V},x)(\psi(v)+b{\bf 1}_{V})\\
&=Y_{V}(\phi(i(u)+a{\bf 1}),x)(\phi(i(v)+b{\bf 1})),\\
\end{align*}
which is what we needed.

In order to establish the uniqueness, consider any other vertex
algebra $F$ with an embedding $j:E' \rightarrow F$ satisfying the same
property as $E''$ and $i$.  Then we get unique vertex algebra
homomorphisms $\phi:E'' \rightarrow F$ and $\xi:F \rightarrow E''$
such that for all $u \in E'$
\begin{align*}
\phi(i(u))&=j(u)\\
\xi(j(u))&=i(u),
\end{align*}
which in turn gives
\begin{align*}
\xi(\phi(i(u)))&=\xi(j(u))=i(u)\\
\text{and}&\\
\phi(\xi(j(u)))&=\phi(i(u))=j(u).
\end{align*}
There is a unique vertex algebra homomorphism $\alpha:E'' \rightarrow
E''$ such that $\alpha(i(u))=i(u)$ for all $u \in E'$ or in other
words such that the diagram
\begin{align*}
\xymatrix{
E' \ar[rd]^i \ar[r]^i & E'' \ar[d]^\alpha\\
&E''}
\end{align*}
commutes.  This unique map $\alpha$ must be the identity map on $E''$,
but also must be $\xi \circ \phi$.  Similarly $\phi \circ \xi$ must be
the identity on $F$.  Therefore $\xi$ and $\phi$ are canonically given
isomorphisms respecting the embeddings.
\end{proof}
\begin{lemma}
Using the notation of Section \ref{sec:main}, we have
$D^{[n]}e_{i}=D^{n}D^{[0]}e_{i}$ whenever $D^{[n]}e_{i} \in M$.
\end{lemma}
\begin{proof}
When acting on $E$ we have
\begin{align*}
Y'(D^{n}D^{[0]}e_{i},x)
=\left(\frac{d}{dx}\right)^{n}Y'(D^{[0]}e_{i},x)
=\left(\frac{d}{dx}\right)^{n}Y(e_{i},x)
=Y'(D^{[n]}e_{i},x),
\end{align*}
where the first equality follows from Lemma \ref{lem:dder} and the
second and third equalities follow from (\ref{eq:Y'}) together with
linearity.  Thus the result follows by Lemma \ref{lem:EtoE'} and Lemma
\ref{lem:Y'inj}.
\end{proof}
Let $(V',Y',D')$ be a $D$-ertex algebra.  Let $(V,Y)$ be a
sub-ertex algebra of $V'$ regarded as merely an ertex algebra.  We
call $\langle V,D' \rangle $ the smallest $D$-ertex algebra
contained in $V'$ and containing $V$ such that the
derivative operator is a restriction of $D'$. Or in
other words, $\langle V,D' \rangle$ is the smallest $D$-ertex
subalgebra of $V'$ which contains $V$.  Linearity and
(\ref{eqfortrunc}) immediately show that $\langle V,D' \rangle$ is the
linear span of vectors of the form $D'^{n}v$ where $n \geq 0$ and $v
\in V$.

Let $(E,Y)$ be an ertex algebra and $(E',Y',D')$ a $D$-ertex algebra.
Let $f:E \rightarrow E'$ be an ertex algebra homomorphism.  We say
that $f$ is {\it $D$-injective} if $\langle f(E),D' \rangle$
is an injective ertex algebra.  
%We say that $f$ is {\it $D$-surjective} if
%$\langle f(E),D' \rangle=E'$.
\begin{theorem}
\label{U2}
Using the notation of Section \ref{sec:main} let $i$ be the obvious
embedding of $E$ into $E'$.  Given any $D$-ertex algebra
$(V,Y_{V},D_{V})$ such that there exists a {\it
$D$-injective} ertex algebra homomorphism, $\psi:E
\rightarrow V$, then there exists a unique $D$-ertex algebra
homomorphism $\phi: E' \rightarrow V$ such that $\phi (i(u))=\psi(u)$
for all $u \in E$, or in other words such that the diagram
%\begin{align*}
%\begin{CD}
%E @>i>> E'\\
%@VV \psi V   @VV \phi V \\
%V @= V.
%\end{CD}
%\end{align*}
\begin{align*}
\xymatrix{
E \ar[rd]^\psi \ar[r]^i & E' \ar[d]^\phi\\
&V}
\end{align*}
commutes.

Furthermore, $E'$ is the unique, up to canonical isomorphism, $D$-ertex
algebra which has a $D$-injective 
%and $D$-surjective 
embedding of $E$
into it and which satisfies the above property.
\end{theorem}
%\begin{remark} \rm
%\label{rem:univ}
%The property in the above theorem is not a universal property because
%the set of injective ertex algebras does not form a category in which
%the morphisms are $D$-injective ertex algebra homomorphisms.
%\end{remark}
\begin{proof}
If $\phi$ exists, then we must have 
\begin{align*}
\phi (D^{[n]}e_{l})
=\phi(D^{n}D^{[0]}e_{l})
=D_{V}^{n}\phi(D^{[0]}e_{l})
=D_{V}^{n}\phi(i(e_{l}))
=D_{V}^{n}\psi(e_{l}),
\end{align*}
for all $D^{[n]}e_{l} \in M$.  Thus $\phi$, if it exists must be
unique.  And furthermore we may use the above to define $\phi$ as the
unique linear map satisfying
\begin{align*}
\phi (D^{[n]}e_{l})
=D_{V}^{n}\psi(e_{l}),
\end{align*}
and such that $\phi (i(u))=\psi(u)$
for all $u \in E$.

We shall next show that $\phi$ is compatible with the
derivative operators.  Consider $D^{[n]}e_{l} \in M$.
If $D^{[n+1]}e_{l} \in M$ then
\begin{align*}
\phi DD^{[n]}e_{l}=\phi D^{[n+1]}e_{l}=D_{V}^{n+1}\psi(e_{l})=D
_{V} \phi D^{[n]}e_{l}.
\end{align*}
And in the case that $D^{[n+1]}e_{l} \notin M$.  Then we may write
\begin{align*}
DD^{[n]}e_{l}=\bar{D}^{[n+1]}e_{l}=\sum_{k \geq 0,l \in
I}a_{k,l}D^{[k]}e_{l},
\end{align*}
where each summand is scalar multiple of an
element of $M$.  Then for $j \in I$ we have
\begin{align*}
Y_{V}(\phi DD^{[n]}e_{l}-&D_{V}\phi D^{[n]}e_{l},x)\psi(e_{j})\\
&=Y_{V}\left(
\phi\left(\sum_{k \geq 0,l \in
I}a_{k,l}D^{[k]}e_{l}\right)
-D_{V}^{n+1}\psi(e_{l}),x\right)\psi(e_{j})\\
&=Y_{V}\left(
\sum_{k \geq 0,l \in I}a_{k,l}
D_{V}^{k}\psi(e_{l})
-D_{V}^{n+1}\psi(e_{l}),x\right)\psi(e_{j})\\
&=\sum_{k \geq 0,l \in I}a_{k,l}
\left(\frac{d}{dx}\right)^{k}
Y_{V}(\psi(e_{l}),x)\psi(e_{j})
-\left(\frac{d}{dx}\right)^{n+1}
Y_{V}(\psi(e_{l}),x)\psi(e_{j})\\
&=\sum_{k \geq 0,l \in I}a_{k,l}
\left(\frac{d}{dx}\right)^{k}
\psi(Y(e_{l},x)e_{j})
-\left(\frac{d}{dx}\right)^{n+1}
\psi(Y(e_{l},x)e_{j})\\
&=\sum_{k \geq 0,l \in I}a_{k,l}
\left(\frac{d}{dx}\right)^{k}
\phi(i(Y(e_{l},x)e_{j}))
-\left(\frac{d}{dx}\right)^{n+1}
\phi(i(Y(e_{l},x)e_{j}))\\
&=\phi
\left(
\sum_{k \geq 0,l \in I}a_{k,l}
\left(\frac{d}{dx}\right)^{k}
Y'(i(e_{l}),x)i(e_{j})
-\left(\frac{d}{dx}\right)^{n+1}
Y'(i(e_{l}),x)i(e_{j})\right)\\
&=\phi
\left(
Y'(\sum_{k \geq 0,l \in I}a_{k,l}
D^{k}
i(e_{l}),x)i(e_{j})
-
Y'(D^{n+1}i(e_{l}),x)i(e_{j})
\right)\\
&=\phi
\left(
Y'(\sum_{k \geq 0,l \in I}a_{k,l}
D^{[k]}
e_{l},x)i(e_{j})
-
Y'(D^{[n+1]}e_{l},x)i(e_{j})
\right)\\
&=\phi
\left(
Y'(\bar{D}^{[n+1]}e_{l},x)i(e_{j})
-
Y'(D^{[n+1]}e_{l},x)i(e_{j})
\right)\\
&=\psi\left(
\bar{Y}(\bar{D}^{[n+1]}e_{l}-D^{[n+1]}e_{l},x)e_{j}\right)\\
&=0.
\end{align*}
Since $E'$ has the $\mathcal{D}$-bracket derivative property and
because $\psi$ is $D$-injective, by linearity this gives
that
\begin{align*}
\phi DD^{[n]}e_{l}=D_{V}\phi D^{[n]}e_{l}.
\end{align*}
Thus by linearity we have that
\begin{align*}
\phi D u =D_{V} \phi u
\end{align*}
for all $u \in E'$.

We next show that $\phi$ is an ertex algebra homomorphism.  For all
$l,k \in I$ we have
\begin{align*}
\phi (Y'(e^{yD}i(e_{k}),x)e^{zD}i(e_{l}))
&=\phi e^{zD}Y'(i(e_{k}),x+y-z)i(e_{l})\\
&=e^{zD_{V}}\phi Y'(i(e_{k}),x+y-z)i(e_{l})\\
&=e^{zD_{V}}\phi (i(Y(e_{k},x+y-z)e_{l}))\\
&=e^{zD_{V}}\psi Y(e_{k},x+y-z)e_{l}\\
&=e^{zD_{V}} Y_{V}(\psi(e_{k}),x+y-z)\psi(e_{l})\\
&=Y_{V}(e^{yD_{V}}\psi(e_{k}),x)e^{zD_{V}}\psi(e_{l})\\
&=Y_{V}(e^{yD_{V}}\phi (i(e_{k})),x)e^{zD_{V}}\phi(i(e_{l}))\\
&=Y_{V}(\phi \left(e^{yD}i(e_{k})\right),x)
\phi\left(e^{zD}i(e_{l}) \right),
\end{align*}
so that the result, except for the uniqueness of $E'$, follows by linearity.

To establish the uniqueness of $E'$, we first note that since $\langle
i(E),D \rangle = E'$ and by Lemma \ref{lem:Y'inj} we have that $i$ is
a $D$-injective 
%and $D$-surjective 
ertex algebra homomorphism.  Let us
say that $(F,Y_{F},D_{F})$ is another such $D$-ertex algebra, with
$D$-injective 
%and $D$-surjective 
ertex algebra embedding of $E$ into
$F$ given by $j$.  We get a $D$-ertex algebra homomorphism $\phi:E'
\rightarrow F$ and a second $D$-ertex algebra homomorphism $\xi:F
\rightarrow E'$.  Moreover, for all $v \in E$ we have
\begin{align*}
\xi(j(v))=i(v)\\
\phi(i(v))=j(v),
\end{align*}
which in turn gives
\begin{align*}
\xi(\phi(i(v)))&=\xi(j(v))=i(v)\\
\text{and}&\\
\phi(\xi(j(v)))&=\phi(i(v))=j(v).
\end{align*}
There is a unique $D$-ertex algebra homomorphism $\alpha:E' \rightarrow
E'$ such that $\alpha(i(v))=i(v)$ for all $v \in E$ or in other
words such that the diagram
\begin{align*}
\xymatrix{
E \ar[rd]^i \ar[r]^i & E' \ar[d]^\alpha\\
&E'}
\end{align*}
commutes.  This unique map $\alpha$ must be the identity map on $E'$,
but also must be $\xi \circ \phi$.  Similarly $\phi \circ \xi$ must be
the identity on $F$.  Therefore $\xi$ and $\phi$ are canonically given
isomorphisms respecting the embeddings.
%Therefore
%\begin{align*}
%\xi(\phi(D^{n}(i(v))))
%=\xi(D_{F}^{n}(\phi(i(v))))
%=D^{n}(\xi(\phi(i(v))))
%=D^{n}(\xi(j(v)))
%=D^{n}i(v),
%\end{align*}
%so that by linearity $\xi \circ \phi:E' \rightarrow E'$ is the
%identity map.  Thus $\xi$ is surjective and $\phi$ is injective.
%Likewise we have
%\begin{align*}
%\phi(\xi(D_{F}^{n}(j(v))))
%=\phi(D^{n}(\xi(j(v)))
%=D_{F}^{n}\phi(\xi(j(v)))
%=D_{F}^{n}\phi(i(v))
%=D_{F}^{n}j(v),
%\end{align*}
%so that by linearity $\phi \circ \xi:F \rightarrow F$ is the identity
%map.  Thus $\phi$ is surjective and $\xi$ is injective.  Then either
%$\phi$ or $\xi$ serves as the isomorphism we need.
\end{proof}

%Let $V$ be a vertex algebra with subset $S$.  Then following the
%notation of \cite{LL} we say that $\langle S \rangle$ is the smallest
%vertex sub-algebra containing $S$.  If $i:E \rightarrow V$ is an ertex
%algebra homomorphism such that $\langle i(E) \rangle =V$ then we say
%that $i$ is {\it vacuum surjective}.

\begin{corollary}
\label{U3}
Using the notation of the second proof of Theorem \ref{main} at the
close of Section \ref{sec:main}, let $i$ be the obvious embedding of
$E$ into $E''$.  Given any vertex algebra $(V,Y_{V},{\bf 1}_{V})$ such
that there exists a {\it $D$-injective} ertex algebra homomorphism,
$\psi:E \rightarrow V$, then there exists a unique vertex algebra
homomorphism $\phi: E'' \rightarrow V$ such that $\phi (i(u))=\psi(u)$
for all $u \in E$, or in other words such that the diagram
%\begin{align*}
%\begin{CD}
%E @>i>> E''\\
%@VV \psi V   @VV \phi V \\
%V @= V.
%\end{CD}
%\end{align*}
\begin{align*}
\xymatrix{
E \ar[rd]^\psi \ar[r]^i & E'' \ar[d]^\phi\\
&V}
\end{align*}
commutes.

Furthermore $E''$ is the unique, up to canonical isomorphism, vertex
algebra with a $D$-injective 
%and vacuum surjective 
embedding of $E$
into it which satisfies the above property.
\end{corollary}
%\begin{remark} \rm
%The above corollary is not a universal property but see Remark \ref{rem:univ}.
%\end{remark}
\begin{proof}
Since every vertex algebra homomorphism is also a $D$-ertex
algebra homomorphism we may argue similarly as in Theorem \ref{U2} to
get that $\phi$ must be uniquely determined on all elements of the
form $D^{[n]}e_{i} \oplus 0$.  Further, since the vacuum must be
preserved, by linearity we have that $\phi$, if it exists, must be
unique.

By Theorem \ref{U2} we get a $D$-ertex algebra homomorphism
$\xi:E' \rightarrow V$ and by Proposition \ref{U1} this in turn gives us
the desired vertex algebra homomorphism.  We next check that
indeed $\phi (i(u))=\psi(u)$.  Let $j$ be the obvious embedding of $E$
into $E'$ and let $k$ be the obvious embedding of $E'$ into $E''$.
Then we have $\phi(k(v))=\xi(v)$ for all $v \in E'$ and
$\xi(j(u)=\psi(u)$ for all $u \in E$.  Therefore
\begin{align*}
\phi(i(u))=\phi(k(j(u)))=\xi(j(u))=\psi(u).
\end{align*}

In order to establish the uniqueness of $E''$ we first note that $i$
is $D$-injective as we noted in the proof of Theorem \ref{U2}.  
%that it is easy to see that $i$ is vacuum surjective.  
Let us say that $F$
with a $D$-injective
%and vacuum surjective 
embedding $j:E \rightarrow F$ is
another vertex algebra with the same property.  Then we get two vertex
algebra homomorphisms $\phi:E'' \rightarrow F$ and $\xi: F \rightarrow
E''$ such that $\phi(i(v))=j(v)$ and $\xi(j(v))=i(v)$ for all $v \in E$.  Thus
\begin{align*}
\phi(\xi(j(v)))=\phi(i(v))=j(v)
\end{align*}
and
\begin{align*}
\xi(\phi(i(v)))=\xi(j(v))=i(v).
\end{align*}
There is a unique vertex algebra homomorphism $\alpha:E'' \rightarrow
E''$ such that $\alpha(i(v))=i(v)$ for all $v \in E$ or in other
words such that the diagram
\begin{align*}
\xymatrix{
E \ar[rd]^i \ar[r]^i & E'' \ar[d]^\alpha\\
&E''}
\end{align*}
commutes.  This unique map $\alpha$ must be the identity map on $E''$,
but also must be $\xi \circ \phi$.  Similarly $\phi \circ \xi$ must be
the identity on $F$.  Therefore $\xi$ and $\phi$ are canonically given
isomorphisms respecting the embeddings.
%Since $i$ and $j$ are vacuum surjective, by Proposition 3.9.3
%\cite{LL} it is easy to see that $\phi \circ \xi:F \rightarrow F$ and
%$\xi \circ \phi: E'' \rightarrow E''$ are the identity maps and the
%result follows easily.
\end{proof}

We recall the notation from Theorem \ref{main} and let $i$ be the
obvious embedding of $E$ into $E''_{r}$.  The following theorem states
that $E''_{r}$ satisfies a universal property.
\begin{theorem}
Using the notation of Theorem \ref{main} let $i$ be the obvious
embedding of $E$ into $E''_{r}$.  Given any vertex algebra $V$ such
that there exists an injective ertex algebra homomorphism $\psi:E
\rightarrow V$, then there exists a unique injective vertex algebra
homomorphism $\phi:E''_{r} \rightarrow V$ such that $\phi
(i(u))=\psi(u)$ for all $u \in E$, or in other words such that the
diagram 
%\begin{align*}
%\begin{CD}
%E @>i>> E'\\
%@VV \psi V   @VV \phi V \\
%V @= V.
%\end{CD}
%\end{align*}
\begin{align*}
\xymatrix{
E \ar[rd]^\psi \ar[r]^i & E''_{r} \ar[d]^\phi\\
&V}
\end{align*}
commutes.

Furthermore, $E''_{r}$ is the unique, up to canonical isomorphism, vertex
algebra equipped with an embedding of $E$ into it and which satisfies the
above property.
\end{theorem}

\begin{proof}
If $\phi$ exists it must be unique because it is determined on a
generating set of $E''_{r}$.  Let $V$ be a vertex algebra with $\psi:E
\rightarrow V$ an embedding.  Then consider the adjoint representation
of $V$, which we call $\xi$.  Consider the vertex algebra $V_{r}$ of
weak vertex operators generated by $\xi(\psi(E))$.  It is clear that
$V_{r}$ is isomorphic to $E''_{r}$ under the unique vertex algebra
isomorphism $j:E''_{r} \rightarrow V_{r}$ satisfying
\begin{align*}
j(i(v))=\xi(\psi(v)),
\end{align*}
for all $v \in E$.  Further, since $V$ is a vertex algebra, $\xi$ is
faithful.  Let
\begin{align*}
\phi(v)=\xi^{-1}(j(v)) ,
\end{align*}
for all $v \in E''_{r}$.  Then $\phi$ is clearly an injective vertex
algebra homomorphism and
\begin{align*}
\phi(i(v))=\xi^{-1}(j(i(v)))=\xi^{-1}(\xi(\psi(v)))=\psi(v),
\end{align*}
for all $v \in E$.  

We now establish the uniqueness of $E''_{r}$. Let us say that $F$,
with an embedding $j:E \rightarrow F$ is another vertex algebra with
the same property.  Then we get two vertex algebra homomorphisms
$\phi:E''_{r} \rightarrow F$ and $\xi: F \rightarrow E''_{r}$ such that
$\phi(i(v))=j(v)$ and $\xi(j(v))=i(v)$ for all $v \in E$.  Thus
\begin{align*}
\phi(\xi(j(v)))=\phi(i(v))=j(v)
\end{align*}
and
\begin{align*}
\xi(\phi(i(v)))=\xi(j(v))=i(v).
\end{align*}
There is a unique vertex algebra homomorphism $\alpha:E''_{r} \rightarrow
E''_{r}$ such that $\alpha(i(v))=i(v)$ for all $v \in E$ or in other
words such that the diagram
\begin{align*}
\xymatrix{
E \ar[rd]^i \ar[r]^i & E''_{r} \ar[d]^\alpha\\
&E''_{r}}
\end{align*}
commutes.  This unique map $\alpha$ must be the identity map on
$E''_{r}$, but also must be $\xi \circ \phi$.  Similarly $\phi \circ
\xi$ must be the identity on $F$.  Therefore $\xi$ and $\phi$ are
canonically given isomorphisms respecting the embeddings.
\end{proof}

\begin{remark} \rm
\label{Drepnotcanon}
One could try to mimic the above to formulate a similar looking
property for $E'_{r}$ from Corollary \ref{cor:main} but because it is
not clear whether $E'_{r}$ is injective or not it is also not clear
whether or not $E'_{r}$ is canonical, at least from this point of
view.
\end{remark}

\noindent {\small \sc Department of Mathematics, Rutgers University,
Piscataway, NJ 08854} 
\\ {\em E--mail
address}: thomasro@math.rutgers.edu
\end{document}